\documentclass[10pt]{amsart}

\usepackage{hyperref}
\usepackage{enumerate}
\usepackage{comment}

\makeatletter

\@namedef{subjclassname@2010}{

  \textup{2010} Mathematics Subject Classification}

\usepackage{amssymb, amsmath,amsthm}
\usepackage{mathrsfs}
\newtheorem{thm}{Theorem}[section]

\newtheorem{prop}[thm]{Proposition}

\newtheorem{cor}[thm]{Corollary}

\newtheorem{lem}[thm]{Lemma}
\theoremstyle{definition}
\newtheorem{rem}{Remark}
\newtheorem{def1}{Definition}

\newtheorem{prob}[thm]{Problem}

\newcommand{\ra}{\rightarrow}

\newcommand{\mc}{\mathcal}

\newcommand{\mb}{\mathbb}
\newcommand{\sg}{\sigma}

\newcommand{\la}{\lambda}

\renewcommand{\ss}{\substack}

\newcommand{\e}{\varepsilon}

\renewcommand{\bar}{\overline}

\begin{document}
\title{On a Goldbach-type problem for the Liouville function}
\author{Alexander P. Mangerel}
\address{Department of Mathematical Sciences, Durham University, Stockton Road, Durham, DH1 3LE, UK}
\email{smangerel@gmail.com}
\maketitle
\begin{abstract}
Let $\lambda$ denote the Liouville function. We show that for all sufficiently large integers $N$, the (non-trivial) convolution sum bound
$$
\left|\sum_{n < N} \lambda(n) \lambda(N-n)\right| < N-1
$$
holds. This (essentially) answers a question posed at the 2018 AIM workshop on Sarnak's conjecture.
\end{abstract}
\section{Introduction}
\subsection{Main result}
The following problem was posed at the 2018 AIM workshop on Sarnak's conjecture.
\begin{prob}[Problem 5.1 of \cite{AIMList}] \label{prob:Rad}
Prove that for every\footnote{Actually, no range for $N$ was given in the problem. It is our presumption that the bound was meant to be shown for any $N \geq 3$.} $N \geq 3$ the sum
$$
\mc{L}_{\la}(N) := \sum_{1 \leq n < N} \lambda(n)\lambda(N-n)
$$
satisfies $|\mc{L}_{\la}(N)| < N-1$. 
\end{prob}
Obviously, the triangle inequality furnishes the trivial bound $|\mc{L}_{\lambda}(N)| \leq N-1$. Thus, the problem is to show \emph{any savings} over this bound. This should be interpreted as an analogue of the binary Goldbach problem for the Liouville function. Indeed, if $N \geq 4$ is even and $\lambda$ is replaced by the prime-supported von Mangoldt function $\Lambda_1(n) := (\log n)1_{n \text{ prime}}$, then proving the existence of primes $p,q$ with $p+q = N$ is equivalent to
$$
\mc{L}_{\Lambda_1}(N) := \sum_{n < N} \Lambda_1(n)\Lambda_1(N-n) > 0,
$$
i.e., showing any improvement over the trivial lower bound $\mc{L}_{\Lambda_1}(N) \geq 0$. \\
Problem \ref{prob:Rad} is far weaker than what we expect to be true regarding the convolution sum $\mc{L}_{\la}(N)$. It is natural to compare the problem at hand with what ought to follow from Chowla's conjecture \cite{Cho}, namely that for all (fixed) $h \geq 1$,
$$
\frac{1}{x} \left|\sum_{n \leq x} \lambda(n) \lambda(n+h)\right| = o(1) \text{ as } x \ra \infty.
$$
On the heuristic basis that the additively coupled values $\lambda(n)$ and $\lambda(N-n)$ ought to also be ``almost orthogonal'' on average, Corr\'{a}di and K\'{a}tai \cite{CorKa} have conjectured that $|\mc{L}_{\la}(N)| = o(N)$ as $N \ra \infty$.  As far as we are aware, the only result in this direction is a theorem of De Koninck, Germ\'{a}n and K\'{a}tai \cite{DKGK}, who proved the conjecture of Corr\'{a}di and K\'{a}tai under the assumption that there are infinitely many Siegel zeroes. Nevertheless, because $N$ is a large shift, recent methods that have proven effective in bounding binary correlations of multiplicative functions seem unsuited to the estimation of the convolution sum $\mc{L}_{\la}(N)$, and thus the problem at hand remains non-trivial unconditionally. \\ 
Our main theorem (essentially) solves Problem \ref{prob:Rad}.
\begin{thm} \label{thm:main}
There exists $N_0 \in \mb{N}$ such that if $N \geq N_0$ then $|\mc{L}_{\la}(N)| < N-1$. 
\end{thm}
\begin{rem}
The proof of Theorem \ref{thm:main} relies on Siegel's theorem on lower bounds for $L(1,\chi)$, where $\chi$ is a quadratic Dirichlet character. Thus, the lower bound $N_0$ is ineffective. See Remark \ref{rem:STBd} below for an indication of what sorts of effective results (up to a possible unique exception) may be proved, in the case of $N$ \emph{prime}, using the Siegel-Tatuzawa theorem \cite{Tat}.
\end{rem}
\subsection{Proof strategy}
We briefly explain our strategy as follows. As we show below (see Lemma \ref{lem:restricttoPrime}), in order to prove Theorem \ref{thm:main} we may restrict to the case in which $N=p^k$ is a prime power, and it is instructive to first consider the case $k = 1$. \\
In this case, it is readily observed that if $|\mc{L}_{\lambda}(p)| = p-1$ in contradiction to the claim, then $\lambda(m)\lambda(p-m)$ is constant over all $1 \leq m < p$, in fact
$$
\lambda(m) \lambda(p-m) = \lambda(p-1)\lambda(1) = \lambda(p-1).
$$
But note that if $\chi_p = \left(\frac{\cdot}{p}\right)$ denotes the Legendre symbol modulo $p$ then the same is true of $\chi_p$:
$$
\chi_p(m) \chi_p(p-m) = \chi_p(-1)\chi_p(m)^2 = \chi_p(p-1).
$$ 
Inspired by this comparison, we seek to show that $\lambda(m) = \chi_p(m)$ in the fundamental domain $[1,p-1]$ for $\chi_p$. Using harmonic analysis over $\mb{Z}/p\mb{Z}$, the problem reduces to understanding the Fourier coefficients of $n \mapsto \lambda(n) 1_{[1,p-1]}(n)$, i.e., exponential sums
$$
S_{\lambda}(\xi) := \sum_{1 \leq n < p} \lambda(n)e(n\xi/p), \quad \xi \pmod{p},
$$
where as usual $e(t) = e^{2\pi i t}$ for $t \in \mb{R}$. The corresponding sums with $\lambda$ replaced by $\chi_p$ are the twisted Gauss sums
$$
\tau(\chi_p,\xi) := \sum_{1 \leq n \leq p} \chi_p(n) e\left(\frac{n\xi}{p}\right) = \chi_p(\xi) \tau(\chi_p,0), \quad \xi \pmod{p}.
$$
In particular, we have the \emph{dilation property} that for each $1 \leq d < p$,
$$
\tau(\chi_p,d\xi) = \chi_p(d)\tau(\chi_p,\xi) \text{ for all } \xi \pmod{p}.
$$
We prove below (see Proposition \ref{prop:dilation}) that whenever $|\mc{L}_\la(p)| = p-1$ a similar dilation property holds for $S_{\lambda}$, i.e., for each $1 \leq d < p$,
\begin{equation}\label{eq:dilByD}
S_{\la}(d\xi) = \lambda(d)S_{\lambda}(\xi) \text{ for all } \xi \pmod{p}.
\end{equation}
The upshot of this is that when $d$ is a primitive root modulo $p$ we may determine all of the sums $S_{\lambda}(\xi)$, $\xi \neq 0$, which turn out to coincide precisely with the twisted Gauss sums given above. In this way, verifying \eqref{eq:dilByD} allows us to determine that $\lambda(n) = \chi_p(n)$ for all $1 \leq n < p$. \\
It turns out that, under the assumption $|\mc{L}_{\la}(p)| = p-1$, proving the dilation property \eqref{eq:dilByD} is equivalent to proving that
$$
\lambda(m)\lambda(m+jp) = +1 \text{ whenever } 1 \leq m < p, \, 0 \leq j < d \text{ and } m \equiv -jp \pmod{d}.
$$
We prove that this property holds for all pairs of $(m,j)$ in question using an iterative argument (see Section \ref{sec:iter}). The rough idea of that argument is to 
$$
\text{replace } \lambda(m)\lambda(m+jp) \text{ by a prescribed sign multiplied by } \lambda(m')\lambda(m'+j'p),
$$
in which $m'\equiv - j'p \pmod{d}$ and $0 \leq j' < j$; crucially, the parameter $j$ has been decremented. Iterating this procedure must eventually result in $j' = 0$, in which case the product on the right-hand side is simply $+1$. We are then able to calculate the original product $\lambda(m)\lambda(m+jp)$ to be $+1$ as well. Our argument may be seen as extending the ``periodic'' behaviour imposed on $\lambda$ by the relation $\lambda(m)\lambda(p-m) = \lambda(p-1)$ for all $m \in \{1,\ldots,p-1\}$, to the larger domain $[1,dp-1]$. \\
Having showed that $\lambda(n) = \chi_p(n)$ for all $n < p$, we return to the general prime power case $N = p^k$, and deduce  
upper bounds both for $p$ and for $k$. First, the bound on $p$ arises from the observation that if $\lambda(n) = \chi_p(n)$ for sufficiently many $n$ then $\chi_p$ is an exceptional character. We obtain a bound of the shape $L(1,\chi_p) \ll_{\e} p^{-1/6+\e}$ for the Dirichlet $L$-function of $\chi_p$ at $s = 1$ (see Proposition \ref{prop:LlamBd}), which is in contradiction to Siegel's theorem \cite[Thm. 5.28(2)]{IK} when $p$ is sufficiently large. This is the source of ineffectivity in our bounds on $N_0$. \\
Next, we note that if $p^k||N$ then the constraint $\lambda(n) = \chi_p(n)$ can be extended to all $n < p^k$ by an inductive argument (see Lemma \ref{lem:extendtopk} below). In particular, this implies that $\chi_p(q) = -1$ at all primes $q < p^k$. On the other hand, using Linnik's theorem \cite[Thm. 18.1]{IK} we can find primes $q < p^L$, with $L$ an effectively computable constant, for which $\chi_p(q) = +1$ necessarily (namely when $q \pmod{p}$ is a quadratic residue). This implies that $k \leq L$.\\
It follows that $N$ must be composed of prime powers $p^k$ with both $k \leq L$ and $p \leq p_0$, so that $N$ itself must be bounded by some $N_0$, as required.
\begin{rem}
One may also ask\footnote{We thank Mark Shusterman for pointing out this problem to us, which he asked in the the MathOverflow post \url{https://mathoverflow.net/questions/307479/goldbachs-conjecture-for-the-liouville-function}.} another natural Goldbach-type problem regarding the Liouville function: given an even integer $N \geq 4$, must there exist $1 \leq a,b \leq N$ with $a + b = N$, such that $\lambda(a) = \lambda(b) = -1$? This is obviously implied by the binary Goldbach conjecture, and therefore a weakening of it. \\
The methods of this paper appear to be far too rigid to address this problem directly. Note however, that even a result of the form
\begin{equation}\label{eq:perturbedAIM}
\left|\sum_{n < N} \lambda(n)\lambda(N-n)\right| < N-g(N),
\end{equation}
where $g: \mb{R} \ra \mb{R}$ is a (sufficiently quickly) increasing function satisfying $g(x) = o(x)$, would suffice to prove the existence of such a pair $(a,b)$. \\
Indeed, suppose otherwise. Then for any $1 \leq n < N$, $(1-\lambda(n))(1-\lambda(N-n)) = 0$. It follows that
$$
0 = \sum_{1 \leq n < N} (1-\lambda(n))(1-\lambda(N-n)) = N-1 -2\sum_{n < N} \lambda(n) + \sum_{n < N} \lambda(n)\lambda(N-n).
$$
We deduce from this and the prime number theorem that, e.g.,
$$
\left|\sum_{n < N} \lambda(n)\lambda(N-n)\right| > N- 2\left|\sum_{n < N} \lambda(n)\right| \geq N-CNe^{-\sqrt{\log N}},
$$
for some absolute constant $C > 0$ and all $N \geq 3$. Thus, the choice $g(x) = Cx e^{-\sqrt{\log x}}$ would suffice to this end. \\
It is natural to ask to what extent the techniques in this paper may be perturbed in order to prove a bound like \eqref{eq:perturbedAIM}. We plan to return to this problem in a future paper.
\end{rem}

\section*{Acknowledgments} 
We would like to thank Oleksiy Klurman, Brad Rodgers and Mark Shusterman for useful comments and encouragement.
\section{Reduction to the prime power case}
Assume for the sake of contradiction that $N \geq 2$ satisfies $|\mc{L}_\la(N)| \geq N-1$. By the triangle inequality, we must have $|\mc{L}_{\la}(N)| = N-1$. Our first lemma shows that we may restrict our attention to prime power values of $N$.
\begin{lem} \label{lem:restricttoPrime}
Let $N \in \mb{N}$ and assume that there is a divisor $d|N$ such that $|\mc{L}_{\la}(d)| < d-1$. Then $|\mc{L}_{\la}(N)| < N-1$. In particular, if $|\mc{L}_{\la}(N)| = N-1$ then $|\mc{L}_{\la}(d)| = d-1$ for all $d | N$. \\
Moreover, if $|\mc{L}_{\la}(N)| = N-1$ then $\mc{L}_{\la}(d) = \lambda(d-1)(d-1)$ and $\lambda(N-1) = \lambda(d-1)$ for all $d|N$.
\end{lem}
\begin{proof}
Write $N = md$. Splitting the sum over $n$ defining $\mc{L}_{\la}(N)$ according to whether $m|n$ or not, the triangle inequality implies
\begin{align*}
|\mc{L}_{\la}(N)| &\leq |\sum_{\ss{n < N \\ m|n}} \lambda(n)\lambda(N-n)| + |\sum_{\ss{n < N \\ m \nmid n}} \lambda(n)\lambda(N-n)| \\
&\leq |\sum_{\ss{k < d}} \lambda(mk)\lambda(m(d-k))| + N-1 - |\{1 \leq n < N : m|n\}| \\
&= |\sum_{k < d} \lambda(k)\lambda(d-k)| + N-1-(d-1) = |\mc{L}_{\la}(d)| + N-d.
\end{align*}
Thus, if $|\mc{L}_{\la}(d)| < d-1$ then $|\mc{L}_{\la}(N)| < d-1 + N-d = N-1$, as required. \\
For the second claim, since each summand of $\mc{L}_{\la}(d)$ is $\pm 1$ and there are $d-1$ terms in its support, we have $|\mc{L}_{\la}(d)| = d-1$ if and only if
$$
\lambda(n)\lambda(d-n) = \lambda(1)\lambda(d-1) = \lambda(d-1) \text{ for all } 1 \leq n < d.
$$
It thus follows that $\mc{L}_{\la}(d) = \lambda(d-1)(d-1)$, as claimed. Finally, since
$$
\lambda(n)\lambda(N-n) = \lambda(N-1) \text{ for all } n < N,
$$
specialising to $n = km$ for $k \in \mb{N}$, we find
$$
\lambda(N-1) = \lambda(km)\lambda(m(d-k)) = \lambda(k)\lambda(d-k) = \lambda(d-1),
$$
as required.
\end{proof}
In view of Lemma \ref{lem:restricttoPrime}, we may analyse the condition $|\mc{L}_{\la}(N)| = N-1$ by considering the implied constraints $|\mc{L}_{\la}(p^j)| = p^j-1$, whenever $p^j|N$. In the next two subsections we will obtain constraints both on the size of $p|N$, as well as the multiplicity $k$ such that $p^k||N$.
\subsection{Bounds on the size of prime divisors of $N$}
In the sequel, write
$$
S_{\lambda}(\xi) := \sum_{1 \leq n < p} \lambda(n) e(n\xi/p), \quad \xi \in \mb{Z}/p\mb{Z}.
$$
For the purposes of illustration let us observe that the condition $|\mc{L}_{\la}(p)| = p-1$ imposes rigid constraints on the exponential sums $S_{\la}(\xi)$. Indeed, we have that
$$
\frac{1}{p}\sum_{\xi \pmod{p}} S_{\lambda}(\xi)^2 = \sum_{1 \leq n,m < p} \lambda(n)\lambda(m) \frac{1}{p}\sum_{\xi \pmod{p}} e\left(\frac{(n+m)\xi}{p}\right) = \sum_{1 \leq n < p} \lambda(n)\lambda(p-n) = \mc{L}_{\la}(p),
$$
since $n+m \equiv 0 \pmod{p}$ with $1 \leq n,m < p$ if and only if $n+m = p$. 
As noted in the proof of Lemma \ref{lem:restricttoPrime}, if $|\mc{L}_{\la}(p)| = p-1$ then 
\begin{equation}\label{eq:convProp}
\lambda(m)\lambda(p-m) = \lambda(p-1)\lambda(1) = \lambda(p-1) \text{ for all } 1 \leq m < p.
\end{equation}
and also $\mc{L}_{\lambda}(p) = \lambda(p-1)(p-1)$. Therefore,
$$
\frac{\lambda(p-1)}{p}\sum_{\xi \pmod{p}} S_{\lambda}(\xi)^2 = p-1 = \frac{1}{p}\sum_{\xi \pmod{p}} |S_{\lambda}(\xi)|^2.
$$
To motivate our forthcoming arguments, we explicitly observe the following relations, which will imply further rigidity in the values of $S_{\la}(\xi)$.
\begin{lem} \label{lem:relations}
Let $p > 2$ be prime and let $1 \leq m < p$. Suppose $|\mc{L}_{\lambda}(p)| = p-1$. Then: 
\begin{enumerate}[(a)]
\item if $m$ is odd then $\lambda(p+m)  = \lambda(m)$,
\item if $m \equiv p \pmod{3}$ then $\lambda(2p+m) = \lambda(m)$, 
\item if $m \equiv 2p \pmod{3}$ then $\lambda(2p-m) = \lambda(p-1)\lambda(m)$. 
\end{enumerate}
\end{lem}
\begin{proof}
(a) If $1 \leq m < p$ is odd then $(p\pm m)/2 \in \mb{Z} \cap [1,p-1]$, and we have
$$
p = \frac{p+m}{2} + \frac{p-m}{2}.
$$
As $|\mc{L}_{\lambda}(p)| = p-1$, using \eqref{eq:convProp} with $m$ replaced by $(p-m)/2$ we get
$$
\lambda(p-1) = \lambda\left(\frac{p+m}{2}\right)\lambda\left(\frac{p-m}{2}\right) = \lambda(p+m)\lambda(p-m).
$$
Since also $\lambda(p-1) = \lambda(m)\lambda(p-m)$, the first claim follows. \\
(b) The argument here is similar: if $m \equiv p \pmod{3}$ then $(2p+m)/3, (p-m)/3 \in \mb{Z} \cap [1,p-1]$ and also 
$$
p = \frac{2p+m}{3} + \frac{p-m}{3},
$$
whence we obtain by \eqref{eq:convProp} that
$$
\lambda(p-1) = \lambda\left(\frac{2p+m}{3}\right)\lambda\left(\frac{p-m}{3}\right) = \lambda(2p+m)\lambda(p-m) = \lambda(p-1)\lambda(2p+m)\lambda(m),
$$
from which the claim follows. \\
(c) Suppose $m \equiv 2p \pmod{3}$. Again, the above idea yields
$$
\lambda(p-1) = \lambda\left(\frac{2p-m}{3}\right)\lambda\left(\frac{p+m}{3}\right) = \lambda(2p-m)\lambda(p+m).
$$
If $m$ is odd then $\lambda(p+m) = \lambda(m)$ by (a), and the claim follows immediately. Otherwise, if $m$ is even then note that since $p > 2$, $p-m$ is odd and hence (a) again yields
$$
\lambda(2p-m) = \lambda(p+(p-m)) = \lambda(p-m) = \lambda(p-1)\lambda(m).
$$
Thus, claim (c) follows in this case as well. 
\end{proof}
The three relations given in Lemma \ref{lem:relations} allows us to deduce the following.
\begin{lem} \label{lem:23dilate}
Let $p > 3$ and assume that $|\mc{L}_{\lambda}(p)| = p-1$. Then 
\begin{equation} \label{eq:dilates}
S_{\lambda}(3\xi) = S_{\lambda}(2\xi) = - S_{\lambda}(\xi) \text{ for all } \xi \pmod{p}.
\end{equation}
In particular, for any $j,k \geq 0$ we have
$$
S_{\lambda}(3^j) = (-1)^j S_{\lambda}(1), \quad S_{\lambda}(2^k) = (-1)^k S_{\lambda}(1).
$$
\end{lem}
\begin{proof}
The second claim follows by induction on $j$ and $k$ from the former, so it suffices to prove \eqref{eq:dilates}. \\
First, note that as $p > 3$ the maps $\xi \mapsto 2\xi$ and $\xi \mapsto 3\xi$ are both bijections on $\mb{Z}/p\mb{Z}$. By Plancherel's theorem,
\begin{equation}\label{eq:ParsDil}
\frac{1}{p}\sum_{\xi \pmod{p}} |S_{\lambda}(m\xi)|^2 = \frac{1}{p} \sum_{\xi' \pmod{p}} |S_{\la}(\xi')|^2 = p-1 \text{ for all } m \in \{1,2,3\}.
\end{equation}
Next, note that
$$
\frac{1}{p}\sum_{\xi \pmod{p}} S_{\lambda}(2\xi) \bar{S}_{\lambda}(\xi) = \sum_{m,n < p}\lambda(n)\lambda(m) \frac{1}{p} \sum_{\xi \pmod{p}} e\left(\xi\frac{2n-m}{p}\right) = \sum_{\ss{m,n < p \\ 2n \equiv m \pmod{p}}} \lambda(m)\lambda(n).
$$
Among $m,n < p$ with $2n \equiv m \pmod{p}$ we have that either $2n = m$ precisely when $m$ is even, or else $2n = m+p$ precisely when $m$ is odd. If $2n = m$ then
$$
\lambda(n)\lambda(m) = \lambda(m/2)\lambda(m) = -\lambda(m)^2 = -1,
$$
while if $2n = p+m$ then by Lemma \ref{lem:relations}(a) we have
$$
\lambda(n)\lambda(m) = \lambda\left(\frac{p+m}{2}\right)\lambda(m) = -\lambda(p+m)\lambda(m) = -1.
$$ 
It follows that
\begin{align*}
\frac{1}{p}\sum_{\xi \pmod{p}} S_{\lambda}(2\xi) \bar{S}_{\lambda}(\xi) &= \sum_{\ss{m < p \\ m \text{ even}}}\lambda(m)\lambda(m/2) + \sum_{\ss{m < p \\ m \text{ odd}}} \lambda(m)\lambda\left(\frac{p+m}{2}\right) \\
&= -\left(\sum_{\ss{m < p \\ m \text{ even}}} 1 + \sum_{\ss{m < p \\ m \text{ odd}}} 1\right) = -(p-1).
\end{align*}
This latter sum being real-valued, it follows from this and \eqref{eq:ParsDil} that
\begin{align*}
\frac{1}{p} \sum_{\xi \pmod{p}}|S_{\lambda}(2\xi) + S_{\lambda}(\xi)|^2 &= \frac{1}{p}\sum_{\xi \pmod{p}} (|S_{\lambda}(\xi)|^2 + |S_{\lambda}(2\xi)|^2 + 2\text{Re}(S_{\lambda}(2\xi)\bar{S}_{\lambda}(\xi))) \\
&= 2(p-1) - 2(p-1) = 0.
\end{align*}
Therefore, $S_{\lambda}(2\xi) = -S_{\lambda}(\xi)$ for all $\xi \pmod{p}$, as claimed. \\
The proof with $2$ replaced by $3$ follows similar lines, using Lemma \ref{lem:relations}. Here, instead we must treat pairs $1 \leq n,m < p$ with $3n \equiv m \pmod{p}$, or equivalently $3n = m+jp$, where $0 \leq j < 3$ and $m+jp \equiv 0 \pmod{3}$ in each case. (Note that as $p > 3$, the residue classes $0,-p,-2p \pmod{3}$ cover all integers $1 \leq m < p$.) Precisely,
\begin{align*}
&\frac{1}{p}\sum_{\xi \pmod{p}} S_{\lambda}(3\xi) \bar{S}_{\lambda}(\xi) = \sum_{m,n < p}\lambda(n)\lambda(m) \frac{1}{p} \sum_{\xi \pmod{p}} e\left(\xi\frac{3n-m}{p}\right) = \sum_{\ss{m, n < p \\ 3n \equiv m \pmod{p}}} \lambda(m)\lambda(n) \\
&= \left(\sum_{\ss{m < p \\ 3|m}} \lambda(m/3)\lambda(m) + \sum_{\ss{m < p \\ m \equiv -2p \pmod{3}}} \lambda(m)\lambda\left(\frac{2p+m}{3}\right) + \sum_{\ss{m < p \\ m \equiv -p \pmod{3}}} \lambda(m)\lambda\left(\frac{p+m}{3}\right)\right).
\end{align*}
By (b) and (c) of Lemma \ref{lem:relations},
\begin{align*}
&\lambda(m) \lambda(m/3) = \lambda(3)\lambda(m)^2 = -1 &\text{ if } 3|m,\\
&\lambda(m)\lambda\left(\frac{2p+m}{3}\right) = \lambda(3)\lambda(m)^2 = -1 &\text{ if } m \equiv -2p \equiv p \pmod{3}, \\
&\lambda(m)\lambda\left(\frac{p+m}{3}\right) = \lambda(m) \lambda(p-1) \lambda\left(\frac{2p-m}{3}\right) = \lambda(3)\lambda(m)^2 = -1 &\text{ if } m \equiv -p \equiv 2p \pmod{3}.
\end{align*}
We thus conclude that
$$
\frac{1}{p}\sum_{\xi \pmod{p}} S_{\lambda}(3\xi) \bar{S}_{\lambda}(\xi) = -\left(\sum_{\ss{m < p \\ 3|m}} 1 + \sum_{\ss{m < p \\ m \equiv -2p \pmod{3}}} 1 + \sum_{\ss{m < p \\ m \equiv - p \pmod{3}}} 1 \right) = -(p-1).
$$
The claim that $S_{\la}(3\xi) = - S_{\la}(\xi)$ for all $\xi \pmod{p}$ now follows as it did with $2\xi$. 
\end{proof}
It is natural, then, to speculate that $2$ or $3$ may be replaced by other primes as well. In fact, we will prove the following more general result in the next section.
\begin{prop}\label{prop:dilation}
Let $p$ be a prime with $|\mc{L}_{\la}(p)| = p-1$, and let $1 \leq d < p$. Then we have
\begin{equation}\label{eq:transform}
S_{\lambda}(d\xi) = \lambda(d) S_{\lambda}(\xi) \text{ for all } \xi \pmod{p}.
\end{equation}
\end{prop}
%
Proposition \ref{prop:dilation} will be beneficial in light of the following result.
\begin{lem} \label{lem:2primroot}
Let $p>2$ be a prime satisfying $|\mc{L}_{\la}(p)| = p-1$. Suppose $2 \leq d < p$ is a primitive root modulo $p$ such that \eqref{eq:transform} holds. 
Then $\lambda(n) = \chi_p(n)$ for all $n < p$. 
\end{lem}
\begin{proof}
As $d$ is a primitive root, every $\xi \not \equiv 0 \pmod{p}$ can be written as $\xi \equiv d^k \pmod{p}$ for some $1 \leq k \leq p-1$, and thus $S_{\la}(\xi) = S_{\la}(d^k)$. By \eqref{eq:transform} and induction, it follows that $S_{\lambda}(d^k) = \lambda(d)^k S_{\lambda}(1)$ for all $k \geq 1$. In particular, $|S_{\la}(\xi)|  = |S_{\la}(1)|$ for all $\xi \not \equiv 0 \pmod{p}$. \\
On the basis of these observations, we first verify that $\lambda(d) = -1$. Since $p$ is odd, 
\begin{align}\label{eq:withdilation}
0 = \sum_{\xi \pmod{p}} S_{\lambda}(\xi) = S_{\lambda}(0) + \sum_{k = 1}^{p-1} S_{\lambda}(d^k) &= S_{\lambda}(0) + S_{\lambda}(1) \sum_{k=1}^{p-1} \lambda(d)^k \nonumber \\
&= S_{\lambda}(0) + (p-1)S_{\la}(1)1_{\lambda(d) = +1}.
\end{align}
Now supposing $\lambda(d) = +1$ then we have $S_{\lambda}(0) = -(p-1)S_{\lambda}(1)$. On the other hand, since $|S_{\lambda}(\xi)| = |S_{\lambda}(1)|$ for all $\xi \not \equiv 0 \pmod{p}$, we find using \eqref{eq:ParsDil} that
$$
p(p-1) = \sum_{\xi \pmod{p}} |S_{\lambda}(\xi)|^2 = |S_{\lambda}(0)|^2 + |S_{\lambda}(1)|^2 (p-1)= |S_{\lambda}(1)|^2((p-1)^2 + p-1) = |S_{\lambda}(1)|^2 p(p-1).
$$
We deduce that $|S_{\lambda}(1)| = 1$, and thus $|S_{\lambda}(0)| = p-1$. But if $p = 3$ we have $|S_{\lambda}(0)| = 0 \neq 2$ since $\lambda(1) = -\lambda(2)$, and for $p > 3$ we have 
$$
|S_{\lambda}(0)| = \left|\sum_{1 \leq n < p} \lambda(n)\right| = \left|\sum_{3 \leq n < p} \lambda(n)\right| \leq p-3 < p-1,
$$
a contradiction. Thus, we conclude that $\lambda(d) = -1$, and it follows further from \eqref{eq:withdilation} that $S_{\lambda}(0) = 0$. \\
Next, as $d$ is a primitive root and $\chi_p$ is a non-principal character, $\chi_p(d) = -1$. Hence, 
$$
S_{\lambda}(d^k) = \lambda(d)^k S_{\lambda}(1) = (-1)^k S_{\lambda}(1) = \chi_p(d)^k S_{\lambda}(1) \text{ for all } k.
$$
Now for each $n < p$, we have using $S_{\lambda}(0) = 0$ that
$$
\lambda(n) = \sum_{m < p} \lambda(m) 1_{m \equiv n \pmod{p}} = \frac{1}{p} \sum_{\xi \pmod{p}} e\left(-\frac{n\xi}{p}\right) \sum_{m < p} \lambda(m) e\left(\frac{m\xi}{p}\right) = \frac{1}{p}\sum_{\ss{\xi \pmod{p} \\ \xi \not \equiv 0 \pmod{p}}} S_{\lambda}(\xi) e\left(-\frac{n\xi}{p}\right).
$$
Reparametrising $\xi \in (\mb{Z}/p\mb{Z})^{\times}$ as $d^k \pmod{p}$ for $1 \leq k \leq p-1$, we get
$$
\lambda(n) = \frac{1}{p}\sum_{k = 1}^{p-1} S_{\lambda}(d^k) e\left(-\frac{nd^k}{p}\right) = \frac{S_{\lambda}(1)}{p} \sum_{k = 1}^{p-1} \chi_p(d)^k e\left(-\frac{nd^k}{p}\right) = \frac{S_{\lambda}(1)}{p} \sum_{\xi \pmod{p}} \chi_p(\xi) e\left(-\frac{n\xi}{p}\right).
$$
Using standard relations for Gauss sums, we readily find that for any $n < p$,
$$
\lambda(n) = \chi_p(n) \cdot \frac{S_{\lambda}(1) \overline{\tau(\chi_p)}}{p},
$$
where, as usual,
$$
\tau(\chi_p) := \sum_{a \pmod{p}} \chi_p(a) e\left(\frac{a}{p}\right).
$$
If we set $n = 1$ then we plainly have $S_{\lambda}(1)\overline{\tau(\chi_p)}=p$, and the claim follows.
\end{proof}
There are clearly two ways in which Lemma \ref{lem:2primroot} may be used. One is to derive information about the Liouville function using corresponding behaviour of Dirichlet characters; this appears hard to do since we only know that $\lambda$ and $\chi_p$ are comparable within the fundamental domain $[1,p-1]$. \\
The other way is to obtain constraints on the behaviour of Dirichlet characters from the rigidity of the Liouville function, in particular at primes. In this direction we deduce the following.
\begin{prop} \label{prop:LlamBd}
There are effectively computable constants $C_1,p_1 > 0$ such that the following holds: if $p\geq p_1$ is a prime for which $|\mc{L}_{\lambda}(p)| = p-1$ then $L(1,\chi_p) \leq C_1(\log p)/p^{1/6}$.
\end{prop}
\begin{rem}\label{rem:STBd}
We highlight here the following (consequence of a) well-known result of Tatuzawa \cite{Tat}: with the exception of at most one prime $\tilde{p}$, given $\delta \in (0,1/12]$ we have for all $p \geq e^{1/\delta}$, $p\neq \tilde{p}$, that
$$
L(1,\chi_p) \geq \frac{3\delta}{5p^{\delta}}.
$$
This allows us to give an effective range (outside of a possible lone exception) for the possible primes $p$ for which $|\mc{L}_{\lambda}(p)| = p-1$. For instance, taking $\delta = 1/12$, the above bound combines with Proposition \ref{prop:LlamBd} to show that when $p \geq \max\{p_1,e^{12}\}$ and $p \neq \tilde{p}$,
$$
p^{1/6} \leq 20 C_1 \log p = 120 C_1 \log(p^{1/6}).
$$
When $p \geq e^{12}$ we have $z := p^{1/6} > 4$. Setting $R := 120 C_1$, we find that $z < R \log z < R z^{1/2}$. We may thus conclude that $p < \max\{p_1,e^{12},R^{12}\}$, provided $p \neq \tilde{p}$. \\
A back-of-the-envelope calculation invoking explicit bounds for (a) the Riemann zeta function on the critical line (see e.g. \cite{Trud}) and (b) the constant in the P\'{o}lya-Vinogradov inequality (see e.g. \cite{BaRa}) yield the constants $C_1 = 12.5$ and $p_1 = e^{24}$ here, which means that any $p \neq \tilde{p}$ with $p > e^{88}$ satisfies $|\mc{L}_{\la}(p)| < p-1$. 
\end{rem}
The proof of Proposition \ref{prop:LlamBd} is based on the following result, which makes more precise an old result due to Linnik and Vinogradov \cite{LiVi}.
\begin{lem}\label{lem:LiVigen}
There are effectively computable constants $p_2,C_2 > 0$ such that if $p \geq p_2$ then
$$
\left|\sum_{n < p} \left(1-\frac{n}{p}\right) \sum_{d|n} \chi_p(d) -  \frac{p}{2} L(1,\chi_p)\right|\leq C_2p^{5/6}\log p.
$$
\end{lem}
\begin{proof}
We follow the strategy of \cite{LiVi}. Expanding the smoothed convolution sum gives
$$
S_p := \sum_{n < p} \left(1-\frac{n}{p}\right) (1\ast \chi_p)(n) = \sum_{d < p} \chi_p(d) \sum_{m < p/d} \left(1-\frac{m}{p/d}\right).
$$
Let $\sg_0 := 1+1/\log p$. Using the Mellin identity
$$
\sum_{n \leq y}\left(1-\frac{n}{y}\right) = \frac{1}{2\pi i} \int_{\sg_0 - i\infty}^{\sg_0+i\infty} \frac{y^s}{s(s+1)} \zeta(s) ds,
$$
we obtain that
$$
S_p = \frac{1}{2\pi i}\int_{\sg_0-i\infty}^{\sg_0+i\infty} \frac{p^s}{s(s+1)} \zeta(s) \left(\sum_{d < p} \frac{\chi_p(d)}{d^s}\right) ds.
$$
Truncating the integral at $\text{Im}(s) = p$ yields an error term of size
$$
\ll p \left(\sum_{d < p} \frac{1}{p}\right) \int_{|t| > p} |\zeta(\sg_0+it)| \frac{dt}{t^2} \ll (\log p)^2. 
$$
We shift the contour to the line $\text{Re}(s) = 1/2$, picking up a residue at the pole $s = 1$, of size
$$
\frac{p}{2} \sum_{d < p} \frac{\chi_p(d)}{d} = \frac{p}{2}L(1,\chi_p) + O\left(\sqrt{p}\log p\right),
$$
the error term arising from partial summation and the P\'{o}lya-Vinogradov inequality. Using the standard bounds (see e.g. \cite[Thms. II.3.8-3.9]{Ten})
$$
|\zeta(\sg + it)| = \begin{cases} O_{\e}\left((2+|t|)^{\tfrac{1}{3}(1-\sg) + \e}\right) &\text{ for } \frac{1}{2} < \sg \leq 1, \\ O(\log(2+|t|)) &\text{ for } \sg \geq 1-1/\log p, 1\leq |t| \leq p, \end{cases}
$$
the horizontal lines at height $\text{Im}(s) = \pm p$ together contribute
$$
\leq \frac{1}{p^2} \int_{1/2}^{\sg_0} |\zeta(u+ip)| p^u\left(\sum_{d < p} \frac{1}{d^u}\right) du \ll_{\e} \frac{1}{p} \int_{1/2}^{1-1/\log p} p^{\frac{1}{3}(1-u)+\e} \frac{du}{1-u} + \frac{1}{p} \int_{1-1/\log p}^{\sg_0} (\log p)^2 du \ll \frac{1}{\sqrt{p}}.
$$
Summarising the above, we so far have
$$
S_p = \frac{p}{2}L(1,\chi_p) + \frac{1}{2\pi i} \int_{1/2-ip}^{1/2+ip} \frac{p^s}{s(s+1)} \zeta(s) \left(\sum_{d < p} \frac{\chi_p(d)}{d^s}\right) ds + O\left(\sqrt{p}\log p\right).
$$
Next, suppose $s = 1/2+it$, with $|t| \leq p$, and set 
$$
S_{\chi_p}(u) := \sum_{n \leq u} \chi_p(n), \quad u \geq 1.
$$ 
Let $1 \leq y\leq p^{3/4}$ be a parameter to be chosen shortly. We have by a dyadic decomposition and partial summation that
\begin{align*}
\left|\sum_{d < p} \frac{\chi_p(d)}{d^s}\right| &\leq \sum_{d \leq y} \frac{1}{\sqrt{d}} + \sum_{y \leq 2^j \leq p} \left|\sum_{2^j \leq d < 2^{j+1}} \frac{\chi_p(d)}{d^{1/2+it}}\right| \\
&\ll y^{1/2} + \sum_{y \leq 2^j \leq p} \left(\frac{|S_{\chi_p}(2^{j+1})|}{2^{j+1}} + \frac{|S_{\chi_p}(2^j)|}{2^j} + (1/2+|t|) \int_{2^j}^{2^{j+1}} \frac{|S_{\chi_p}(u)|}{u^{3/2}} du\right).
\end{align*}
Applying the P\'{o}lya-Vinogradov inequality $\max_{1 \leq u \leq p} |S_{\chi_p}(u)| \ll \sqrt{p}\log p$, we obtain the upper bound
\begin{align*}
&\ll y^{1/2} + \sqrt{p}(\log p)\sum_{y \leq 2^j \leq p} \left(\frac{1}{2^{j}} + (1/2+|t|) 2^{-j/2}\right) \ll y^{1/2} + (1+|t|) \frac{\sqrt{p} \log p}{\sqrt{y}}.
\end{align*}
Inserting this together with the subconvexity bound
$$
|\zeta(1/2+it)| \ll (|t|+2)^{1/6} \log(2+|t|), \quad |t| \leq p,
$$
(see e.g., \cite[Cor. II.3.7]{Ten}) and setting $y := p^{2/3} (\log p)^2$, we obtain
\begin{align*}
&\left|\frac{1}{2\pi i} \int_{1/2-ip}^{1/2+ip} \frac{p^s}{s(s+1)} \zeta(s) \left(\sum_{d < p} \frac{\chi_p(d)}{d^s}\right) ds \right| \\
&\ll p^{1/2}y^{1/2} \int_{|t| \leq p} \frac{(|t|+2)^{1/6} \log(|t| + 2) }{(2+|t|)^2}dt + \frac{p(\log p)}{\sqrt{y}} \int_{|t| \leq p} \frac{(2+|t|)^{7/6}\log(|t|+2)}{(2+|t|)^2} dt \\
&\ll p^{5/6} \log p.
\end{align*}
Collecting the above estimates yields
$$
S_p = \frac{p}{2}L(1,\chi_p) + O\left(p^{5/6}\log p\right),
$$
as claimed.
\end{proof}
\begin{proof}[Proof of Proposition \ref{prop:LlamBd}]
Suppose $|\mc{L}_{\la}(p)| = p-1$. Since $p$ is prime we can find a primitive root $2\leq d < p$ modulo $p$, and by Proposition \ref{prop:dilation}, $d$ satisfies \eqref{eq:transform} for all $\xi \pmod{p}$. Lemma \ref{lem:2primroot} therefore implies that $\la(n) = \chi_p(n)$ for all $n < p$. \\
Since $1 \ast \lambda$ is the indicator function for the set of perfect squares, by Lemma \ref{lem:LiVigen} we find that
$$
\frac{p}{2} L(1,\chi_p) + O(p^{5/6}\log p) = \sum_{n < p} \left(1-\frac{n}{p}\right) \sum_{d|n} \lambda(d) = \sum_{n < p} \left(1-\frac{n}{p}\right) (1\ast \lambda)(n) < \sum_{m^2 < p} 1 < \sqrt{p}.
$$
We deduce from this that $L(1,\chi_p) \ll p^{-1/2} + (\log p)p^{-1/6} \ll (\log p)p^{-1/6}$, as claimed. 
\end{proof}
\subsection{Bounds on the multiplicity of prime divisors of $N$}
We next extend Lemma \ref{lem:2primroot} to higher powers of $p$.
\begin{lem}\label{lem:extendtopk}
Let $p$ be a prime and assume that $|\mc{L}_{\la}(p^k)| = p^k-1$ for some $k \geq 1$. Then $\lambda(n) = \chi_p(n)$ for all $n< p^k$ with $p \nmid n$.
\end{lem}
\begin{proof}
By Lemma \ref{lem:restricttoPrime}, we know that $|\mc{L}_{\la}(p^j)|=p^j-1$ and $\lambda(p^j-1) = \lambda(p-1)$ for all $1 \leq j \leq k$. In the case $j = 1$, combining Propositions \ref{prop:dilation} and \ref{lem:2primroot}, we see that $\lambda(n) = \chi_p(n)$ for all $n < p$. We now extend this to show that $\lambda(n) = \chi_p(n)$ for all $n < p^k$ with $p \nmid n$. \\
To do this, we prove by induction on $1 \leq j \leq k$ that $\chi_p(q) = -1$ for all primes $q < p^j$, $q \neq p$. The outcome of this induction is that $\lambda(q) = \chi_p(q)$ for all $q < p^k$, $q \neq p$, so that by complete multiplicativity we obtain $\lambda(n) = \chi_p(n)$ for all $n < p^k$ with $p \nmid n$. The base case $j = 1$ follows immediately from $\lambda(n) = \chi_p(n)$ for all $1 \leq n < p$, which we have just established. \\
Assume therefore that $\chi_p(q) = -1$ for all $q < p^i$, $q \neq p$, for all $1 \leq i < j$. We next prove that this is the case for all $q < p^j$, $q \neq p$, or equivalently that $\chi_p(q) = - 1$ for all primes $p^{j-1} < q < p^j$. Note that by complete multiplicativity, the inductive assumption actually implies that $\lambda(n) = \chi_q(n)$ for all $n$ that satisfy $P^+(n) < p^{j-1}$ with $p \nmid n$. Now, we know that 
$$
\lambda(n)\lambda(p^j-n) = \lambda(p^j-1) = \lambda(p-1)
$$
for all $n < p^j$ with $p \nmid n$.
We observe furthermore that for any $n < p^j$ with $p \nmid n$, $P^+(n/P^+(n)) < p^{j-1}$, since either $P^+(n) < p^{j-1}$ or else $n$ has precisely one prime factor $p^{j-1} < q < p^j$, and necessarily $n/q < p$. It follows by the inductive assumption that whenever $1 \leq n < p^j$ with $p \nmid n$,
\begin{align*}
\lambda(p-1) &= \lambda(n)\lambda(p^j-n) = \lambda(P^+(n))\lambda(P^+(p^j-n)) \lambda\left(\frac{n}{P^+(n)}\right)\lambda\left(\frac{p^j-n}{P^+(p^j-n)}\right) \\
&= \chi_p\left(\frac{n}{P^+(n)}\right)\chi_p\left(\frac{p^j-n}{P^+(p^j-n)}\right) = \chi_p(P^+(n)) \chi_p(P^+(p^j-n)) \chi_p\left(n\right)\chi_p\left(p^j-n\right) \\
&= \chi_p(P^+(n)) \chi_p(P^+(p^j-n))\chi_p(p-1) = \chi_p(P^+(n)) \chi_p(P^+(p^j-n))\lambda(p-1).
\end{align*}
We deduce, therefore, that 
\begin{equation}\label{eq:chipRel}
\chi_p(P^+(n)) = \chi_p(P^+(p^j-n)) \text{ for all } 1 \leq n < p^j, \, p \nmid n,
\end{equation}
which we will presently use to show that $\chi_p(q) = -1$ for all $p^{j-1} < q < p^j$. \\
Let us enumerate the primes in the interval $(p^{j-1},p^j)$ as $\{q_1,\ldots,q_R\}$. We prove by a second induction on $1 \leq r \leq R$ that $\chi_p(q_r) = -1$. In the base case $r = 1$, observe that if $n_1 := \lfloor p^j/q_1\rfloor q_1$, which is coprime to $p$, then as $n_1/q_1 < p$ we have $P^+(n_1) = q_1$ (recalling that $j>1$ here). Moreover,
$$
p^j - n_1 = p^j - q_1\left\lfloor \frac{p^j}{q_1} \right\rfloor = q_1 \left\{\frac{p^j}{q_1}\right\} < q_1.
$$
Thus, $P^+(p^j-n_1) < q_1$, and since $q_1$ is the minimal prime $> p^{j-1}$ (and $p \nmid p^j-n_1$), we have $P^+(p^j-n_1) < p^{j-1}$. By our inductive assumption on $j$, we have $\chi_p(P^+(p^j-n_1)) = -1$, and hence by \eqref{eq:chipRel}, $\chi_p(q_1) = \chi_p(P^+(n_1)) = -1$. \\
Now supposing that $\chi_p(q_s) = -1$ for all $1 \leq s \leq r$, we find by the same construction that if $n_{r+1} := \lfloor p^j/q_{r+1}\rfloor q_{r+1}$ then $P^+(n_{r+1}) = q_{r+1}$ whereas 
$$
P^+(p^j-n_{r+1}) = P^+\left(q_{r+1}\left\{\frac{p^j}{q_{r+1}}\right\}\right) < q_{r+1}.
$$ 
It follows that $P^+(p^j-n_{r+1})$ is either $q_s$ for some $1 \leq s \leq r$, or else it is $< p^{j-1}$. Either by our inductive assumption on $r$ or on $j$, we obtain $\chi_p(P^+(p^j-n_{r+1})) = -1$, and thus $\chi_p(q_{r+1}) = -1$ as well. By induction on $r$, we obtain $\chi_p(q_r) = -1$ for all $1 \leq r \leq R$, and thus for all primes $p^{j-1} < q < p^j$. By induction on $j$, the claim follows.
\end{proof}
\begin{cor} \label{cor:primPowBd}
There are effectively computable constants $p_3$ and $L$ such that:
\begin{enumerate}
\item if $p \geq p_3$ and $|\mc{L}_{\la}(p^k)| = p^k-1$ then $k \leq 5$, and
\item if $p < p_3$ and $|\mc{L}_{\la}(p^k)|= p^k-1$ then $k \leq L$.
\end{enumerate}
\end{cor}
\begin{proof}
Suppose $|\mc{L}_{\la}(p^k)| = p^k-1$. By the previous lemma we obtain that $\chi_p(n) = \lambda(n)$ for all $n < p^k$, and in particular $\chi_p(q) = -1$ for all primes $q < p^k$. \\
By Heath-Brown's version of Linnik's theorem \cite{HB}, there is an effectively computable constant $c > 0$ and a prime $q < c p^{5.5}$ such that $q \equiv 1 \pmod{p}$. If $q < p^k$ then we obtain the contradiction $-1 = \chi_p(q) = \chi_p(1) = +1$. We therefore conclude that $p^k < q < cp^{5.5}$. This implies that either $k \leq 5$, or else $p< p_3$, with $p_3 := c^2$. Moreover, in the latter case $p < p_3$ we also have that $2^{k-5.5} < c$, so that $k \leq L$ for $L := 5.5 + \tfrac{\log c}{\log 2}$.
\end{proof}
\begin{proof}[Proof of Theorem \ref{thm:main}] Suppose $N \geq 2$ satisfies $|\mc{L}_\la(N)| = N-1$. By Lemma \ref{lem:restricttoPrime}, we have $|\mc{L}_{\la}(p^k)| = p^k-1$ for each prime power $p^k||N$, and in particular $|\mc{L}_{\la}(p)| = p-1$ for each prime $p|N$. \\
By Proposition \ref{prop:LlamBd}, there is a(n effectively computable) constant $C > 0$ such that
$$
L(1,\chi_p) \leq C(\log p)/p^{1/6} \text{ for each } p|N
$$ 
(taking $C\geq C_1$ as needed to cover the range $p < p_1$). But by Siegel's theorem \cite[Thm. 5.28(2)]{IK}, this is impossible as soon as $p \geq p_0$, where $p_0$ is a(n ineffective) constant. It follows that all primes $p|N$ satisfy $p < p_0$, as claimed. \\
Furthermore, by Corollary \ref{cor:primPowBd} there is a(n effectively computable) constant $L > 0$ such that if $p < p_0$ and $p^k||N$ then $k \leq L$. It follows that 
$$
N = \prod_{p^k|| N} p^k \leq \left(\prod_{p < p_0} p\right)^L =: N_0.
$$
Therefore, if $N > N_0$ then $|\mc{L}_{\la}(N)| < N-1$, as claimed.
\end{proof}
\begin{rem}
The ineffectivity inherent in our main result is due to the possible presence of Siegel zeros, which renders Siegel's theorem ineffective. This seems compelling in that, typically, Siegel zeros are \emph{useful} for establishing non-trivial estimates about correlations of the Liouville function. For example, in regards to the convolution sums treated in this paper, De Koninck, Germ\'{a}n and K\'{a}tai proved \cite{DKGK} that if $(p_k)_k$ is a sequence of exceptional moduli, and if $\eta_k := (1-\beta_k) \log p_k$, then for any $N \in (p_k^{10}, p_k^{c\log\log(1/\eta_k)}]$ (albeit with the condition $(N,p_k) = 1$),
$$
\left|\frac{1}{N}\sum_{n < N} \lambda(n)\lambda(N-n)\right| \leq \frac{1}{\log\log(1/\eta_k)} + o_{k \ra \infty}(1).
$$
However, besides the restriction $(N,p_k) = 1$ not suiting our purposes here, our results are limited by the fact that $\lambda(n) = \chi_{p_k}(n)$ only in the limited range $1 \leq n < p_k$. It is unclear whether the result of \cite{DKGK}, for example, which requires a broader range in which $\lambda$ behaves periodically in some sense, may be applicable in the present circumstances in order to render our theorem effective.
\end{rem}

\section{Proof of Proposition \ref{prop:dilation}} \label{sec:iter}
In this section we prove Proposition \ref{prop:dilation}. As we show later, it suffices to consider the case when $d$ is prime. Having shown the cases $d = 2$ and $3$ in Lemma \ref{lem:23dilate}, we focus here on $d > 3$. 
\subsection{An iterative argument for $d=q$ prime}
As previously, let $p > 3$ be a prime with $|\mc{L}_{\la}(p)| = p-1$. Let $3 < q < p$ be an odd prime. We wish to show that 
$$
S_{\lambda}(q\xi) = -S_{\lambda}(\xi) \text{ for all } \xi \pmod{p}.
$$ 
As in the proof of Lemma \ref{lem:23dilate}, it suffices to show that
\begin{align*}
-(p-1) &= \frac{1}{p} \sum_{\xi \pmod{p}} S_{\lambda}(q\xi) \bar{S}_{\la}(\xi) = \sum_{m,n < p} \lambda(m)\lambda(n) \frac{1}{p}\sum_{\xi \pmod{p}} e\left(\xi\frac{qn-m}{p}\right) \\
&= \sum_{\ss{m,n < p \\ qn \equiv m \pmod{p}}} \lambda(m)\lambda(n).
\end{align*}
As before, we split the set of $1 \leq m,n < p$ according to the choice of $0 \leq j < q$ for which $qn = m+jp$; in each case we have $m \equiv -jp \pmod{q}$, and thus we must show that
$$
-(p-1) = \sum_{j = 0}^{q-1} \sum_{\ss{m < p \\ m \equiv -jp \pmod{q}}} \lambda(m) \lambda\left(\frac{m+jp}{q}\right) = -\sum_{j = 0}^{q-1} \sum_{\ss{m < p \\ m \equiv -jp \pmod{q}}} \lambda(m) \lambda\left(m+jp\right).
$$
Equivalently, it is our goal to prove that for every $0 \leq j < q$ and $1 \leq m < p$ we have
\begin{equation}\label{eq:equivcond}
\lambda(m)\lambda(m+jp) = +1 \text{ whenever } m \equiv -jp \pmod{q}.
\end{equation}
For $1 \leq r \leq q-1$ we define the sets
$$
\mc{A}_{q,r} := \{(m,j) \in \{1,\ldots,p-1\} \times \{0,\ldots,q-1\} : \frac{pq}{r+1} < m+jp < \frac{pq}{r}\}, \quad \mc{A}_{q,q} := \{(m,0) : 1 \leq m < p\}.
$$
Note that the sets $\{\mc{A}_{q,r}\}_{1 \leq r \leq q}$ partition the set of all pairs $(m,j) \in \{1,\ldots,p-1\} \times \{0,\ldots,q-1\}$, in view of the following observations:
\begin{enumerate}[(a)]
\item for each such pair, $1 \leq m+jp < pq$, and therefore must either satisfy $1 \leq m+jp < p$ (so $j = 0$), or else $pq/(r+1) \leq m+jp < pq/r$ for some $1 \leq r \leq q-1$;
\item as $p$ and $q$ are prime we can never have $m+jp = pq/(r+1)$ for any $1 \leq r \leq q-1$ unless $r = q-1$, but in this case $m + jp = p$ is not solvable with $1 \leq m < p-1$;
\item if $(m,j) \notin \mc{A}_{q,r}$ for all $1 \leq r \leq q-1$ then $1 \leq m+jp < p$, equivalently, $j = 0$ and $(m,0) \in \mc{A}_{q,q}$.
\end{enumerate}
For each $1 \leq r \leq q-1$ define a map $\psi_r$ on pairs $(m,j) \in \mc{A}_{q,r}$ via
$$
\psi_r(m,j) := \left(\left\lceil\frac{rm}{p}\right\rceil p - r m, q-jr-\left\lceil \frac{rm}{p}\right\rceil\right),
$$
where, as usual, given $t \in \mb{R}$ we denote by $\lceil t \rceil$ the least integer $k \geq t$.
\begin{lem}\label{lem:psir}
Let $(m,j) \in \mc{A}_{q,r}$ for some $1 \leq r \leq q-1$, and set $(m',j') := \psi_r(m,j)$. The following properties hold: 
\begin{enumerate}[(a)]
\item $m'+j'p \equiv 0 \pmod{q}$ whenever $m+jp \equiv 0 \pmod{q}$,
\item $m' = \{-rm/p\}p = p\left(1-\{rm/p\}\right)$,
\item $(m',j') \in \bigcup_{r+1 \leq s \leq q} \mc{A}_{q,s}$, and
\item if $m+jp \equiv 0 \pmod{q}$ then $\lambda(m+jp) = \lambda(r)\lambda(p-1) \lambda(m'+j'p)$.
\end{enumerate}
In fact, (d) may be rewritten as
\begin{equation}\label{eq:withPer}
\lambda(m)\lambda(m+jp) = \left[\lambda(rm) \lambda\left(p\left\{\frac{rm}{p}\right\}\right)\right] \lambda(m')\lambda(m'+j'p).
\end{equation}
\end{lem}
\begin{proof}
(a) We observe that
\begin{equation} \label{eq:afterIter}
m'+j'p = \left(\left\lceil\frac{rm}{p}\right\rceil p - r m\right) + \left(q-jr-\left\lceil \frac{rm}{p}\right\rceil\right)p = p(q-jr)-rm = pq-r(m+jp),
\end{equation}
so that if $m+jp \equiv 0\pmod{q}$ then $m'+j'p \equiv 0 \pmod{q}$ as well.\\
(b) Note that since $r,m < p$, $rm/p \notin \mb{Z}$. Whenever $\alpha \notin \mb{Z}$ we have 
$$
\left\lceil \alpha \right\rceil = \alpha + 1-\{\alpha\}= \alpha + \{-\alpha\}, 
$$
so we therefore conclude that
$$
m' = p\left(\frac{rm}{p} + \left\{-\frac{rm}{p}\right\}\right) - rm = p\left\{-\frac{rm}{p}\right\} = p\left(1-\left\{\frac{rm}{p}\right\}\right),
$$
as required.\\
(c) Using $pq/(r+1) < m+jp < pq/r$ together with \eqref{eq:afterIter}, we deduce that
\begin{align*}
m'+j'p &< pq-r\cdot \frac{pq}{r+1} = pq\left(1-\frac{r}{r+1}\right) = \frac{pq}{r+1} \\
m'+j'p &> pq-r\cdot \frac{pq}{r} = 0.
\end{align*}
Together with (b), the latter bound implies that $1 \leq  m' < p$. Furthermore,
$$
j' = q-\left(jr + \left\lceil \frac{rm}{p}\right\rceil \right) \geq q-1 - \frac{r}{p}(jp+m) > q-1 - \frac{r}{p} \cdot \frac{pq}{r} = -1,
$$
so as $j'$ is an integer with $j' > -1$, and $\lceil rm/p\rceil \geq 1$, we have $0 \leq j' < q$. As $0 < m'+j'p < pq/(r+1)$, it follows that $(m',j') \in \mc{A}_{q,s}$ for some $r+1 \leq s \leq q$, as required. \\
%
%
(d) Since $(m,j) \in\mc{A}_{q,r}$ and $q|(m+jp)$, we see that 
$$
1 \leq \frac{m+jp}{q} < \frac{1}{q} \cdot \frac{pq}{r} = \frac{p}{r}.
$$
It follows that $(rm+rjp)/q \in \mb{Z} \cap [1,p-1]$, and so using \eqref{eq:convProp},
\begin{align*}
\lambda(m+jp) &= \lambda(qr) \la\left(\frac{rm+rjp}{q}\right) = \lambda(qr)\lambda(p-1) \la\left(p-\frac{rm+rjp}{q}\right) = \lambda(qr)\lambda(p-1)\la\left(\frac{(q-rj)p - rm}{q}\right) \\
&= \lambda(r)\lambda(p-1) \la\left(\left(\left\lceil \frac{rm}{p} \right\rceil p - rm\right) + p\left(q-rj-\left\lceil \frac{rm}{p}\right\rceil\right)\right) \\
&= \lambda(r)\la(p-1) \la(m'+j'p),
\end{align*}
as claimed. \\
To prove \eqref{eq:withPer}, it suffices to note using (b) that
$$
\lambda(m') = \lambda(p-p\{rm/p\}) = \lambda(p-1)\lambda(p\{rm/p\}),
$$
after which the identity follows immediately from (c) upon multiplying both sides by $\lambda(m)$.
\end{proof}
The upshot of Lemma \ref{lem:psir}(c) is that $\psi_r$ maps $\mc{A}_{q,r}$ to a set of pairs $(m',j')$ for which $m'+j'p$ has strictly decreased (and in particular $(m',j')$ belongs to $\mc{A}_{q,r'}$ where $r' > r$). We see therefore that by iteratively composing maps $\psi_r$, $r < q$, we must eventually find an image pair in $\mc{A}_{q,q}$, i.e., where the $j$ component is $0$. \\
With this in mind, we introduce the following definition.
\begin{def1} We say that the \emph{signature} of a pair $(m,j)$, $1 \leq m < p$ and $0 \leq j < q$, is the tuple $(r_1,\ldots,r_k)$ such that $1 \leq r_1 < r_2 < \cdots < r_k < q$, with 
$$
\psi_{r_k} \circ \cdots \circ \psi_{r_1}(m,j) = (\tilde{m},0) \in \mc{A}_{q,q},
$$ 
for some $1 \leq \tilde{m} < p$. (Here, we implicitly have that the indices $r_i$ are determined such that $(m,j) \in \mc{A}_{q,r_1}$, $\psi_{r_1}(m,j) \in \mc{A}_{q,r_2}$, $\psi_{r_2} \circ \psi_{r_1}(m,j) \in \mc{A}_{q,r_3}$, and so on.)
\end{def1} 
\begin{lem} \label{lem:iteration}
Let $(m,j)$ have signature $(r_1,\ldots,r_k)$. Then
$$
\lambda(m)\lambda(m+jp) = \prod_{i = 0}^{k-1} \lambda(m_ir_{i+1})\lambda\left(p\left\{\frac{r_{i+1}m_i}{p}\right\}\right),
$$
where we have set
$$
(m_0,j_0) := (m,j), \quad (m_{i+1} , j_{i+1}) := \psi_{r_{i+1}}(m_i,j_i) \text{ for } 0 \leq i \leq k-1.
$$
\end{lem}
\begin{proof}
By iteratively invoking \eqref{eq:withPer}, we obtain
\begin{align*}
\lambda(m)\lambda(m+jp) &= \lambda(m_0)\lambda(m_0+j_0p) = \left[ \lambda(m_0 r_{1}) \lambda\left(p\left\{\frac{r_{1}m_0}{p}\right\}\right) \right] \lambda(m_{1})\lambda(m_{1}+j_{1}p) \\
&= \prod_{i = 0}^1\left[ \lambda(m_i r_{i+1}) \lambda\left(p\left\{\frac{r_{i+1}m_i}{p}\right\}\right) \right] \lambda(m_2)\lambda(m_2+j_2p) \\
&= \cdots = \prod_{i = 0}^{k-1}\left[ \lambda(m_i r_{i+1}) \lambda\left(p\left\{\frac{r_{i+1}m_i}{p}\right\}\right) \right] \lambda(m_k)\lambda(m_k+j_kp).
\end{align*}
By the definition of signature, we have $j_k = 0$, and thus $\lambda(m_k)\lambda(m_k+j_kp) = \lambda(m_k)^2 = +1$. The claim follows.
\end{proof}
\subsection{Periodicity via dilation}
In connection with Lemma \ref{lem:psir} we next show the following lemma, which shows that if $S_{\la}$ satisfies the dilation property in Proposition \ref{prop:dilation} with $d = r < p$ then $\la$ exhibits mod $p$ periodicity in $[1,rp-1]$.
\begin{lem} \label{eq:inductiveperiod}
Assume that $1 \leq r < p$ satisfies
\begin{equation}\label{eq:dilbyr}
S_{\lambda}(r\xi) = \lambda(r)S_{\lambda}(\xi) \text{ for all } \xi \pmod{p}.
\end{equation}
Then for any $1 \leq m < p$, $\lambda(p\{rm/p\}) = \lambda(r)\lambda(m)$.
\end{lem}
\begin{proof}
Note that by making the bijective change of variables $\xi \mapsto r^{-1} \xi \pmod{p}$ and rearranging, \eqref{eq:dilbyr} yields
$$
S_{\la}(r^{-1}\xi) = \lambda(r)S_{\la}(\xi) \text{ for all } \xi \pmod{p}.
$$
From this, we derive that
\begin{align*}
0 &= S_{\lambda}(r^{-1}\xi) -\lambda(r)S_{\lambda}(\xi) = \sum_{n < p} \lambda(n) e(nr^{-1}\xi/p) - \sum_{m < p} \la(mr) e(m\xi/p) \\
&= \sum_{m < p} e\left(m\xi/p\right) \left(\sum_{\ss{n < p \\ n \equiv rm \pmod{p}}} \lambda(n) - \lambda(mr)\right).
\end{align*}
Since this holds for all $\xi \pmod{p}$ we deduce that for all $1 \leq m < p$,
\begin{equation}\label{eq:invcong}
\sum_{\ss{n < p \\ n \equiv rm \pmod{p}}} \lambda(n) = \lambda(m)\la(r).
\end{equation}
On the other hand, the condition $n \equiv rm \pmod{p}$ with $1 \leq n < p$ is equivalent to $n = p \{rm/p\}$. 
Combining these two facts, we deduce that
$$
\lambda(p\{rm/p\}) = \lambda(m)\lambda(r),
$$
as claimed.
\end{proof}
We are now in a position to prove Proposition \ref{prop:dilation}.
\begin{proof}[Proof of Proposition \ref{prop:dilation}]
We proceed by induction on $d$. When $d = 1$ there is nothing to prove. Thus, assume that for every $1 \leq d < q$ the equation 
\begin{equation}\label{eq:transform1}
S_{\lambda}(d\xi) = \lambda(d) S_{\lambda}(\xi) \text{ holds for all } \xi \pmod{p}.
\end{equation}
We thus must establish \eqref{eq:transform1} for $d = q$ (provided $q < p$). \\
First, we observe that if $q$ is composite then writing $q = ab$ with $1 < a,b < q$ we have that for any $\xi \pmod{p}$,
$$
S_{\lambda}(q\xi) = S_{\lambda}(a(b\xi)) = \lambda(a)S_{\lambda}(b\xi) = \lambda(a)\lambda(b) S_{\lambda}(\xi) = \lambda(q)S_{\lambda}(\xi),
$$
as required. Hence, the claim holds whenever $q$ is composite. \\
Thus, we may assume that $q$ is prime. Assuming the induction hypothesis, we see that for any $1 \leq r < q$ (necessarily coprime to $p$) we have
$$
S_{\lambda}(r\xi) = \lambda(r) S_{\lambda}(\xi) \text{ for all } \xi \pmod{p}.
$$
By Lemma \ref{eq:inductiveperiod}, we see that for any $1 \leq m < p$,
\begin{equation}\label{eq:periodic}
\lambda\left(p\left\{rm/p\right\}\right) = \lambda(r)\lambda(m) = \lambda(rm),
\end{equation}
a fact that we will use momentarily. \\
As discussed above, in order to prove \eqref{eq:transform1} holds with $d = q$ it suffices to show that 
$$
\lambda(m)\lambda(m+jp) = +1 \text{ for all } 0 \leq j < q,  \, 1 \leq m < p \text{ with } m \equiv -jp \pmod{q}.
$$
Now given $1 \leq m < p$ and $0 \leq j < q$ let $(r_1,\ldots,r_k)$ denote the signature of $(m,j)$, recalling that $1 \leq r_1 < \cdots < r_k < q$. By Lemma \ref{lem:iteration}, we have
$$
\lambda(m)\lambda(m+jp) = \prod_{i = 0}^{k-1} \lambda(m_i r_{i+1})\lambda\left(p\left\{\frac{r_{i+1}m_i}{p}\right\}\right).
$$
But since $r_{i+1} < q$ for all $0 \leq i \leq k-1$, we know that \eqref{eq:periodic} holds with each $r = r_{i+1}$. Thus, every factor in the right-hand product is simply $+1$, and hence $\lambda(m)\lambda(m+jp) = +1$, as required. Hence, \eqref{eq:transform1} holds when $q$ is prime as well. \\
The inductive claim therefore follows in all cases, and so by induction, the proof is complete.
%
\end{proof}

\bibliographystyle{plain}
\bibliography{LGBib.bib}

\end{document}